\documentclass[11pt]{amsart}
\usepackage[dvipsnames,usenames]{color}
\usepackage[colorlinks=true, urlcolor=NavyBlue, linkcolor=NavyBlue, citecolor=NavyBlue]{hyperref}
\usepackage{graphicx}
\usepackage{epsfig}
\usepackage[latin1]{inputenc}
\usepackage{amsmath}
\usepackage{amsfonts}
\usepackage{amssymb}
\usepackage{amsthm}
\usepackage{amscd}
\usepackage{verbatim}
\usepackage{subfigure}
\usepackage{stmaryrd}

\usepackage{tikz}
\usetikzlibrary{arrows}
\usepackage{verbatim}
\usepackage{pinlabel}
	
\newtheorem{theorem}{Theorem}[section]
\newtheorem{Theorem}{Theorem}
\newtheorem{lemma}[theorem]{Lemma}
\newtheorem{proposition}[theorem]{Proposition}

\newtheorem{definition}[theorem]{Definition}

\def\Z{\mathbb{Z}}

\def\F{\mathbb{F}}

\def\CFKi{CFK^{\infty}}

\title{On the concordance genus of topologically slice knots}

\subjclass[2009]{}
\author{Jennifer Hom}

\address{Department of Mathematics, Columbia University, New York, NY 10027
\newline\indent{\tt hom@math.columbia.edu}}

\numberwithin{equation}{section}

\begin{document}

\begin{abstract}
The concordance genus of a knot $K$ is the minimum Seifert genus of all knots smoothly concordant to $K$. Concordance genus is bounded below by the $4$-ball genus and above by the Seifert genus. We give a lower bound for the concordance genus of $K$ coming from the knot Floer complex of $K$. As an application, we prove that there are topologically slice knots with $4$-ball genus equal to one and arbitrarily large concordance genus. 
\end{abstract}

\maketitle

\section{Introduction}

The \emph{concordance genus} of a knot $K$, $g_c(K)$, is the minimum genus of all knots smoothly concordant to $K$. The concordance genus is bounded below by the $4$-ball genus and above by the genus; that is,
\[ g_4(K) \leq g_c(K) \leq g(K). \]
Note that taking the connected sum with a slice knot does not change the value of $g_c$, but increases the genus. In this manner, the gap between $g_c(K)$ and $g(K)$ can be made arbitratily large. For many knots, $g_4(K)=g_c(K)$. For example, a consequence of the Milnor conjecture, first proved by Kronheimer and Mrowka \cite{KronMrowka}, is that $g_4(K)=g_c(K)=g(K)$ for torus knots. 

In \cite[Problem 14]{Gordonproblems}, Gordon asks if $g_4(K)=g_c(K)$ in general. Nakanishi \cite{Nakanishi} answered the question in the negative, using Alexander polynomials to show that the gap between $g_4(K)$ and $g_c(K)$ can be arbitrarily large. The more subtle question of whether there are algebraically slice knots for which the gap between $g_4(K)$ and $g_c(K)$ can be arbitrarily large was answered by Livingston in \cite{Livingstonconcordancegenus}, where he used Casson-Gordon invariants to find algebraically slice knots with $4$-ball genus equal to one and arbitrarily large concordance genus. 

Neither the Alexander polynomial nor Casson-Gordon invariants suffice to extend these results to topologically slice knots. In this paper, we give a lower bound for $g_c(K)$ coming from the knot Floer complex of $K$, and use this bound to give a family of topologically slice knots with smooth $4$-ball genus equal to one and arbitrarily large concordance genus.

To a knot $K$ in $S^3$, Ozsv\'ath and Szab\'o \cite{OSknots}, and independently Rasmussen \cite{R}, associate a $\Z \oplus \Z$-filtered chain complex, $\CFKi(K)$, whose filtered chain homotopy type is an invariant of $K$. Associated to this chain complex are several concordance invariants; in this paper, we focus on the invariant $\varepsilon(K)$, a $\{-1, 0, 1\}$-valued invariant defined in \cite{Homsmooth}, and to a lesser extent, the invariant $\tau(K)$, defined in \cite{OS4ball}. Both $\varepsilon$ and $\tau$ are defined by studying certain natural maps on homology induced by inclusions and projections of appropriate subquotient complexes of $\CFKi(K)$.

We say that two $\Z \oplus \Z$-filtered chain complexes, $C_1$ and $C_2$, are \emph{$\varepsilon$-equivalent} if
\[ \varepsilon(C_1 \otimes C_2^*)=0,\]
where $C^*$ denotes the dual of $C$.
We say that two knots, $K_1$ and $K_2$, are \emph{$\varepsilon$-equivalent} if their knot Floer complexes are $\varepsilon$-equivalent, that is, if
\[\ \varepsilon \big(\CFKi(K_1) \otimes \CFKi(K_2)^* \big)=0. \]
As seen in the following theorem, $\varepsilon$-equivalence is closely related to concordance:
\begin{Theorem}[{\cite{Homsmooth}}]
\label{thm:epsilonequivalent}
If two knots are concordant, then they are $\varepsilon$-equivalent.
\end{Theorem}

\noindent We define the \emph{breadth} of a $\Z \oplus \Z$-filtered chain complex $C$, $b(C)$, to be
\[ b(C)=\textup{max} \{ j \ | \ H_*(C(0, j)) \neq 0 \}, \]
where $C(i, j)$ denotes the $(i, j)$-graded summand of the associated graded complex. Recall from \cite[Theorem 1.2]{OSgenusbounds} that
\[ g(K)=b(\CFKi(K)). \]
The invariant $\gamma(K)$ is defined to be the minimum breadth of all filtered chain complexes $\varepsilon$-equivalent to $\CFKi(K)$:
\[ \gamma(K)= \textup{min} \{ b(C) \ | \ \varepsilon(\CFKi(K) \otimes C^*)=0 \}.\]

\begin{Theorem}
\label{thm:gamma}
The invariant $\gamma(K)$ gives a lower bound on the smooth concordance genus of $K$; that is,
\[g_c(K) \geq \gamma(K).\]
\end{Theorem}

\noindent At this first glance, this may seem like an intractable invariant, as the set of chain complexes $\varepsilon$-equivalent to $\CFKi(K)$ is infinite. However, in many situations, there are tractable numerical invariants associated to the $\varepsilon$-equivalence class of $K$ giving lower bounds for $\gamma(K)$, and hence also for $g_c(K)$. In this next theorem, we use these bounds to prove a result concerning the concordance genus of a family of topologically slice knots.

Let $D$ denote the (positive, untwisted) Whitehead double of the right-handed trefoil, and let $K_{p,q}$ denote the $(p, q)$-cable of $K$, where $p$ indicates the longitudinal winding and $q$ the meridional winding. We write $-K$ to denote the reverse of the mirror of $K$.

\begin{Theorem}
\label{thm:theknots}
Let $K_p=D_{p, 1} \# -D_{p-1, 1}$. Then $K_p$ is topologically slice with $g_4(K_p)=1$ and $g_c(K_p)\geq p$.
\end{Theorem}

\noindent In \cite[Theorem 1.5]{Livingstonconcordancegenus}, Livingston constructs algebraically slice knots with $4$-ball genus equal to one and arbitrarily large concordance genus. However, his proof relies on Casson-Gordon invariants, and so his examples are not topologically slice. He also remarks on the inherent challenge in bounding the concordance genus: one must show that the given knot is not concordant to any knot in the infinite family of knots with genus less than a given $N$. The invariant $\gamma$ can help significantly in this regard. Moreover, the invariant $\gamma$ can give useful bounds on the concordance genus of topologically slice knots, while the techniques of \cite{Livingstonconcordancegenus} cannot.

\vspace{.5cm}
\noindent \textbf{Organization.} In Section \ref{sec:bounding}, we recall the necessary properties of Heegaard Floer homology and knot Floer homology, and use them to prove Theorem \ref{thm:gamma}. In Section \ref{sec:theknots}, we apply those results to give a family of topologically slice knots with $4$-ball genus one and arbitrarily large concordance genus.

We work with coefficients in $\F=\Z/2\Z$ throughout. Unless otherwise stated, we work in the smooth category.

\vspace{.5cm}
\noindent \textbf{Acknowledgements.} I would like to thank Chuck Livingston and Peter Horn for many helpful conversations.

\section{Bounding the concordance genus}
\label{sec:bounding}

We recall the basic definitions of knot Floer homology, assuming that the reader is familiar with these invariants; for an expository overview, we suggest \cite{OSsurvey}. In this paper, we concern ourselves primarily with the algebraic properties of the invariant.

To a knot $K \subset S^3$, Ozsv\'ath-Szab\'o \cite{OSknots}, and independently Rasmussen \cite{R}, associate $\CFKi(K)$, a $\Z$-graded, $\Z$-filtered freely generated chain complex over the ring $\F[U, U^{-1}]$, where $U$ is a formal variable. The filtered chain homotopy type of $\CFKi(K)$ is an invariant of the knot $K$. The differential does not decrease the $U$-exponent, and the $U$-exponent (more precisely, the negative of the $U$-exponent) induces a second $\Z$-filtration, giving $\CFKi(K)$ the structure of a $\Z \oplus \Z$-filtered chain complex. The ordering on $\Z \oplus \Z$ is given by $(i, j) \leq (i', j')$ if $i \leq i'$ and $j \leq j'$.

This chain complex is freely generated over $\F [U, U^{-1}]$ by tuples of intersection points in a doubly pointed Heegaard diagram for $S^3$ compatible with the knot $K$. Each generator $x$ comes with a homological, or Maslov grading, $M(x)$, and an Alexander filtration, $A(x)$. The differential, $\partial$, 
decreases the Maslov grading by one, and respects the Alexander filtration; that is,
\[ M(\partial x) = M(x) -1 \qquad \textup{and} \qquad A(\partial x ) \leq A(x).\]
Multiplication by $U$ shifts the Maslov grading by two and the Alexander filtration by one:
\[ M(U \cdot x) = M(x) -2 \qquad \textup{and} \qquad A(U \cdot x) = A(x) -1. \]
It is often convenient to graphically represent this complex in the $(i, j)$-plane, where the $i$-axis corresponds to $-(U\textup{-exponent})$, and the $j$-axis corresponds to the Alexander filtration. The Maslov grading is suppressed from this picture. A generator $x$ is placed at $(0, A(x))$, and a element of the form $U^n \cdot x$ is placed at $(-n, A(x)-n)$.

Given a $\Z \oplus \Z$-filtered chain complex $C$ and $S \subset \Z \oplus \Z$, we write $C\{ S \}$ to denote the set of elements in the plane whose $(i, j)$-coordinates are in $S$ together with the arrows between them. If $S$ has the property that $(i, j) \in S$ implies that $(i', j') \in S$ for all $(i', j') \leq (i, j)$, then $C\{S\}$ is a subcomplex of $C$. We write $C(i, j)$ to denote the subquotient complex with coordinates $(i, j)$, that is, $C \{ (i, j) \}$.

The $\Z$-filtered complex $\widehat{CFK}(K)$ is the subquotient complex consisting of the $j$-axis, i.e., $C\{ i \leq 0 \}/C\{ i<0 \}$. The homology of the associated graded object of $\widehat{CFK}(K)$ is $\widehat{HFK}(K)$. The groups $\widehat{HFK}(K)$ can themselves be viewed as a chain complex, with the differential induced by the higher order, i.e., non-filtration preserving, differentials on $\widehat{CFK}(K)$. Moreover, up to filtered chain homotopy equivalence, $\widehat{HFK}(K)$ is a basis over $\F[U, U^{-1}]$ for $\CFKi(K)$. Choosing $\widehat{HFK}(K)$ as a basis for $\CFKi(K)$ has the advantage that it is \emph{reduced}; that is, the differential strictly lowers the filtration. Graphically, this means that each arrow will point strictly downward or to the left (or both).

We have the following chain homotopy equivalences \cite[Theorem 7.1 and Section 3.5]{OSknots}:
\begin{align*}
\CFKi(K_1 \# K_2) &\simeq \CFKi(K_1) \otimes_{\F[U, U^{-1}]} \CFKi(K_2) \\
\CFKi(-K) &\simeq \CFKi(K)^* 
\end{align*}
where $\CFKi(K)^*$ denotes the dual of $\CFKi(K)$, i.e., $\textup{Hom}_{\F[U, U^{-1}]}(\CFKi(K), \F[U, U^{-1}])$.

To fully exploit the richness of the invariant $\CFKi(K)$, it is helpful to study certain induced maps on homology. For example, the Ozsv\'ath-Szab\'o concordance invariant $\tau$ is defined in \cite{OS4ball} to be
\[ \tau(K) = \textup{min} \{ s \ | \ \iota : C\{i=0, j \leq s \} \rightarrow C\{ i=0 \} \textup{ induces a non-trivial map on homology} \}, \] 
where $\iota$ is the natural inclusion of chain complexes. Note that $H_*(C\{ i=0\})\cong \widehat{HF}(S^3) \cong \F$. The invariant $\tau(K)$ provides a lower bound on the $4$-ball genus of $K$, and gives a surjective homomorphism from the smooth concordance group to the integers \cite{OS4ball}.

More recently, the $\{-1, 0, 1\}$-valued concordance invariant $\varepsilon(K)$ has been defined in \cite{Homsmooth}. To define $\varepsilon$, one first considers the map on homology, $F_*$, induced by the chain map
\[ F: C\{ i=0 \} \rightarrow C\{ \textup{min}(i, j-\tau)=0 \} \]
where $\tau=\tau(K)$, and the chain map consists of quotienting by $C\{ i=0, j< \tau \}$ followed by the inclusion of $C \{ i=0, j \geq \tau \}$ into $C\{ \textup{min}(i, j-\tau)=0 \}$. Similarly, we consider the map $G_*$, induced by
\[ G: C\{ \textup{max}(i, j-\tau)=0 \} \rightarrow C\{ i=0 \}, \]
the composition of quotienting by $C \{ i < 0, j=\tau \}$ and including $C \{ i=0, j \leq \tau \}$ into $C\{ i=0 \}$.

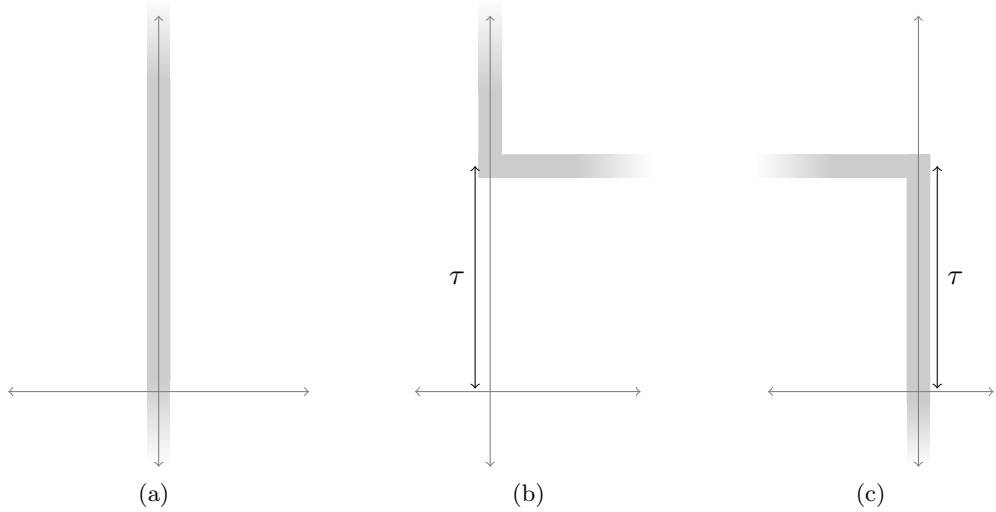
\begin{figure}[htb!]
\vspace{5pt}
\centering
\subfigure[]{
\begin{tikzpicture}
	\filldraw[black!20!white] (-0.15, 0.15) rectangle (0.15, 4.15);
	\shade[top color=white, bottom color=black!20!white] (-0.15, 4.15) rectangle (0.15, 5.25);
	\shade[top color=white, top color=black!20!white] (-0.15, -1) rectangle (0.15, 0.15);
	\begin{scope}[thin, gray]
		\draw [<->] (-2, 0) -- (2, 0);
		\draw [<->] (0, -1) -- (0, 5);
	\end{scope}
\end{tikzpicture}
}
\hspace{25pt}
\subfigure[]{
\begin{tikzpicture}
	\filldraw[black!20!white] (-0.15, 2.85) rectangle (0.15, 4);
	\filldraw[black!20!white] (-0.15, 2.85) rectangle (1.15, 3.15);
	\shade[top color=white, bottom color=black!20!white] (-0.15, 4) rectangle (0.15, 5.25);
	\shade[top color=white, left color=black!20!white] (1.15, 2.85) rectangle (2.15, 3.15);
	\begin{scope}[thin, gray]
		\draw [<->] (-1, 0) -- (2, 0);
		\draw [<->] (0, -1) -- (0, 5);
	\end{scope}
	\draw [<->] (-0.2, 0.04) -- (-0.2, 3) node [midway, left] {$\tau$};
\end{tikzpicture}
}
\hspace{25pt}
\subfigure[]{
\begin{tikzpicture}
	\filldraw[black!20!white] (-0.15, -0.15) rectangle (0.15, 3.15);
	\filldraw[black!20!white] (-1.15, 2.85) rectangle (0.15, 3.15);
	\shade[top color=white, top color=black!20!white] (-0.15, -1) rectangle (0.15, -0.15);
	\shade[top color=white, right color=black!20!white] (-2.15, 2.85) rectangle (-1.15, 3.15);
	\begin{scope}[thin, gray]
		\draw [<->] (-2, 0) -- (1, 0);
		\draw [<->] (0, -1) -- (0, 5);
	\end{scope}
	\draw [<->] (0.25, 0.04) -- (0.25, 3) node [midway, right] {$\tau$};
\end{tikzpicture}
}

\caption{Left, the subquotient complex $C \{ i=0 \}$. Center, the subquotient complex $C\{ \textup{min} (i, j-\tau)=0\}$. Right, the subquotient complex $C\{ \textup{max} (i, j-\tau)=0\}$.}
\end{figure}

\begin{definition}
The invariant $\varepsilon$ is defined in terms of $F_*$ and $G_*$ as follows:
\begin{itemize}
	\item $\varepsilon(K)=1$ if $F_*$ is trivial (in which case $G_*$ is necessarily non-trivial).
	\item $\varepsilon(K)=-1$ if $G_*$ is trivial (in which case $F_*$ is necessarily non-trivial).
	\item $\varepsilon(K)=0$ if $F_*$ and $G_*$ are both non-trivial.
\end{itemize}
\end{definition}
\noindent See \cite[Section 3]{Homsmooth} for details.

Two knots $K_1$ and $K_2$ are \emph{$\varepsilon$-equivalent} if
\[ \varepsilon(K_1 \# -K_2)=0.\]
Concordant knots are $\varepsilon$-equivalent \cite[Theorem 2]{Homsmooth}.

\begin{proof}[Proof of Theorem \ref{thm:gamma}]
The proof that $\gamma(K)$ gives a lower bound on concordance genus is an immediate consequence of the definition of $\gamma(K)$, as follows.
By Theorem \ref{thm:epsilonequivalent}, any two concordant knots are $\varepsilon$-equivalent. Since  $g(K)=b(\CFKi(K))$ by  \cite[Theorem 1.2]{OSgenusbounds} and
\[ \gamma(K)= \textup{min} \{ b(C) \ | \ C \textup{  is $\varepsilon$-equivalent to } \CFKi(K) \}, \]
it follows immediately that
\[ g_c(K) \geq \gamma(K).\]
\end{proof}

Further invariants are defined in \cite[Section 6]{Homsmooth}. Suppose $\varepsilon(K)=1$, and consider the map on homology $H_s$ induced by the chain map
\[ C\{ i=0 \} \rightarrow C\{ \textup{min} (i, j-\tau)=0, i \leq s\}, \]
where $s$ is a non-negative integer, and the map consists of quotienting by $C\{ i=0, j < \tau \}$, followed by inclusion. When $s$ is sufficiently large, the map $H_s$ is trivial since $\varepsilon(K)=1$, while when $s=0$, it is not difficult to see that the map $H_s$ is non-trivial. Thus, one can define
\[a_1= \textup{min} \{ s \ | \ H_s \textup{ is trivial} \}. \]
Going even further, consider the map  on homology $H_{a_1, s}$ induced by
\[ C\{ i=0 \} \rightarrow C \big\{ \{ \textup{min} (i, j-\tau)=0, i \leq a_1\} \cup \{ i=a_1, \tau-s \leq j < \tau\} \big\}, \]
where the map consists of quotienting by $C\{ i=0, j < \tau \}$, followed by inclusion. Define
\[a_2= \textup{min} \{ s \ | \ H_{a_1, s} \textup{ is non-trivial} \}. \]
The set  $\{ s \ | \ H_{a_1, s} \textup{ is non-trivial} \}$ may be empty -- there is no reason why the map $H_{a_1, s}$ must be non-trivial for any $s$ -- in which case the invariant $a_2(K)$ is undefined.

\begin{lemma}[{\cite[Lemma 6.2]{Homsmooth}}]
\label{lem:basis}
Let $a_1=a_1(K)$ and let $a_2=a_2(K)$ be well-defined. Then there exists a basis $\{x_i\}$  over $\F[U, U^{-1}]$ for $\CFKi$ with basis elements $x_0$, $x_1$, and $x_2$ with the property that
\begin{itemize}
	\item \label{it:a_1} There is a horizontal arrow of length $a_1$ from $x_1$ to $x_0$.	
	\item There are no other horizontal or vertical arrows to or from $x_0$.
	\item There are no other horizontal arrows to or from $x_1$.
	\newcounter{enumi_saved}
	\setcounter{enumi_saved}{\value{enumi}}
	\item \label{it:a_2} There is a vertical arrow of length $a_2$ from $x_1$ to $x_2$.
	\item There are no other vertical arrows to or from $x_1$ or $x_2$.
\end{itemize}
\end{lemma}
\noindent See Figure \ref{fig:basis}.

\begin{figure}[htb!]
\vspace{5pt}
\centering
\subfigure[]{
\begin{tikzpicture}
	\filldraw[black!20!white] (-0.15, 2.85) rectangle (0.15, 4);
	\filldraw[black!20!white] (-0.15, 2.85) rectangle (1.15, 3.15);
	\shade[top color=white, bottom color=black!20!white] (-0.15, 4) rectangle (0.15, 5.25);
	\begin{scope}[thin, gray]
		\draw [<->] (-1, 0) -- (3, 0);
		\draw [<->] (0, -1) -- (0, 5);
	\end{scope}
	\draw [<->] (-0.2, 0.04) -- (-0.2, 3) node [midway, left] {$\tau$};
	\draw [<->] (0.05, 3.2) -- (1, 3.2) node [midway, above] {$s$};
\end{tikzpicture}
}
\hspace{25pt}
\subfigure[]{
\begin{tikzpicture}
	\filldraw[black!20!white] (-0.15, 2.85) rectangle (0.15, 4);
	\filldraw[black!20!white] (-0.15, 2.85) rectangle (1.15, 3.15);
	\shade[top color=white, bottom color=black!20!white] (-0.15, 4) rectangle (0.15, 5.25);
	\filldraw[black!20!white] (0.85, 0.85) rectangle (1.15, 3.15);
	\begin{scope}[thin, gray]
		\draw [<->] (-1, 0) -- (3, 0);
		\draw [<->] (0, -1) -- (0, 5);
	\end{scope}
	\draw [<->] (-0.2, 0.04) -- (-0.2, 3) node [midway, left] {$\tau$};
	\draw [<->] (0.05, 3.2) -- (1, 3.2) node [midway, above] {$a_1$};
	\draw [<->] (1.2, 1) -- (1.2, 3) node [midway, right] {$s$};	
\end{tikzpicture}
}
\hspace{25pt}
\subfigure[]{
\begin{tikzpicture}
	\filldraw[black!20!white] (-0.15, 2.85) rectangle (0.15, 4);
	\filldraw[black!20!white] (-0.15, 2.85) rectangle (1.15, 3.15);
	\shade[top color=white, bottom color=black!20!white] (-0.15, 4) rectangle (0.15, 5.25);
	\filldraw[black!20!white] (0.85, 0.85) rectangle (1.15, 3.15);
	\begin{scope}[thin, gray]
		\draw [<->] (-1, 0) -- (3, 0);
		\draw [<->] (0, -1) -- (0, 5);
	\end{scope}
	\filldraw (0, 3) circle (2pt) node[] (x_0){};
	\filldraw (1, 3) circle (2pt) node[] (x_1){};
	\filldraw (1, 1) circle (2pt) node[] (x_2){};
	\draw [very thick, <-] (x_0) -- (x_1);
	\draw [very thick, <-] (x_2) -- (x_1);
	\node [left] at (x_0) {$x_0$};
	\node [right] at (x_1) {$x_1$};
	\node [right] at (x_2) {$x_2$};
\end{tikzpicture}
}
\caption{Left, the complex $C\{ \textup{min} (i, j-\tau)=0, i \leq s\}$. Center, the complex $C \big\{ \{ \textup{min} (i, j-\tau)=0, i \leq a_1\} \cup \{ i=a_1, \tau-s \leq j < \tau\} \big\}$. Right, part of the basis in Lemma \ref{lem:basis}.}
\label{fig:basis}
\end{figure}
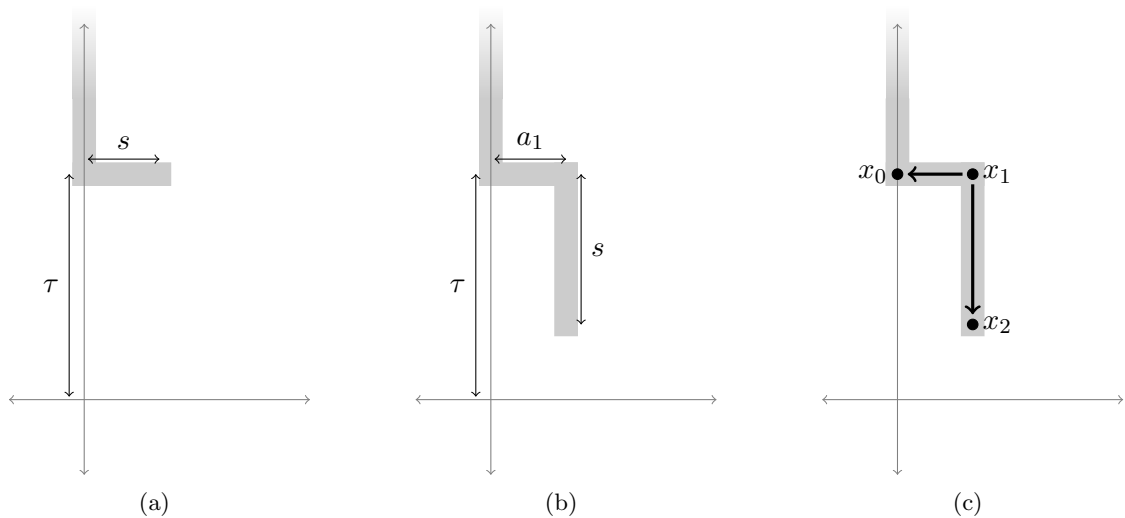

The numbers $a_1$ and $a_2$ are invariants of $\varepsilon$-equivalence \cite[Lemma 6.1]{Homsmooth}. We recall the proof here. If $K_1$ and $K_2$ are $\varepsilon$-equivalent, then $\varepsilon(K_1 \# -K_2)=0$ and by \cite[Lemma 3.3]{Homcables}, there exists a basis for $\CFKi(K_1 \# -K_2)$ with a distinguished basis element $x_0$ with no incoming or outgoing vertical or horizontal arrows. Similarly, there exists a basis for $\CFKi(K_2 \# -K_2)$ with a distinguished basis element $y_0$. The knot $K_1 \# -K_2 \# K_2$ is $\varepsilon$-equivalent to $K_1$ and $K_2$, and we may compute $a_1(K_1 \# -K_2 \# K_2)$ and $a_2(K_1 \# -K_2 \# K_2)$ by considering either 
\[ \{x_0\} \otimes \CFKi(K_2) \qquad \textup{or} \qquad \CFKi(K_1) \otimes \{y_0\}. \]
The former gives us $a_1(K_2)$ and $a_2(K_2)$, and the latter gives us $a_1(K_1)$ and $a_2(K_1)$, completing the proof.

At times, it may be difficult to compute $\gamma(K)$ directly, but we can bound it using the invariants $\tau(K)$, $a_1(K)$, and $a_2(K)$.

\begin{lemma}
\label{lem:gamma}
Suppose that $\varepsilon(K)=1$, and $a_2(K)$ is defined. Then
\[ \gamma(K) \geq |\tau(K)-a_1(K)-a_2(K)|. \]
\end{lemma}

\begin{proof}
From the basis found in Lemma \ref{lem:basis} and the fact that $\tau$, $a_1$, and $a_2$ are invariants of $\varepsilon$-equivalence, it follows that 
\[ H_*\Big(C\big(0, \tau(K)-a_1(K)-a_2(K) \big) \Big) \neq 0, \]
for any complex $C$ that is  $\varepsilon$-equivalent to $\CFKi(K)$.
Using the various symmetry properties of $\CFKi(K)$ \cite[Section 3.5]{OSknots}, it follows that
\[ H_*\big(C(0, |\tau(K)-a_1(K)-a_2(K)|\big) \neq 0, \]
as well. This implies that $b(C) \geq |\tau(K)-a_1(K)-a_2(K)|$ for any $C$ that is $\varepsilon$-equivalent to $\CFKi(K)$, giving the desired bound.
\end{proof}

\section{The knots $D_{p, 1} \# -D_{p-1, 1}$}
\label{sec:theknots}

Let $D$ denote the (positive, untwisted) Whitehead double of the right-handed trefoil. Let $K_{p,q}$ denote the $(p, q)$-cable of $K$, where $p$ indicates the longitudinal winding and $q$ the meridional winding. We will study various properties of the family of knots 
\[ D_{p, 1} \# -D_{p-1, 1},  \quad p>1.\]
 Since the Alexander polynomial of $D$ is equal to one, by Freedman \cite{Freedman} $D$ is topologically slice. Hence the $(p, 1)$-cable of $D$ is topologically concordant to the underlying pattern torus knot, which is unknotted. It follows that the knot $D_{p, 1} \# -D_{p-1, 1}$ is topologically slice.

In the following lemma, we will show that these knots are never smoothly slice.

\begin{lemma}
\label{lem:4genus1}
The smooth $4$-ball genus of the knot $D_{p, 1} \# -D_{p-1, 1}$ is equal to one.
\end{lemma}

\begin{proof}
A genus $p$ Seifert surface for $D_{p, 1}$ can be built from $p$ parallel copies of a genus one Seifert surface for $D$, and $p-1$ half-twisted bands connecting them. Likewise, we may build a genus $p-1$ Seifert surface for $-D_{p-1, 1}$. Connecting these two Seifert surfaces together with a band yields a genus $2p-1$ Seifert surface $F$ for $D_{p, 1} \# -D_{p-1, 1}$. The slice knot $D_{p-1, 1} \# -D_{p-1, 1}$ sits on $F$, and furthermore bounds a subsurface of genus $2p-2$. We may perform surgery on $F$ in $B^4$ along $D_{p-1, 1} \# -D_{p-1, 1}$, yielding a genus one slice surface for $D_{p, 1} \# -D_{p-1, 1}$.

By \cite{HeddenWhitehead}, $\tau(D)=1$, and by \cite[Theorem 1.2]{HeddencablingII} (cf. \cite[Theorem 1]{Homcables}), it follows that $\tau(D_{p, 1})=p$. Therefore, $\tau(D_{p, 1} \# -D_{p-1, 1})=1$, which is a lower bound on the $4$-ball genus of the knot \cite{OS4ball}. Since this bound can be realized, it follows that $g_4(D_{p, 1} \# -D_{p-1, 1})=1$.
\end{proof}

To bound the concordance genus of $K_p=D_{p, 1} \# -D_{p-1, 1}$, we consider its knot Floer complex. We do this using the tools of \cite{Homsmooth} together with the bordered Floer homology package of Lipshitz, Ozsv\'ath, and Thurston \cite{LOT}, as applied to cables by Petkova \cite{Petkova}.

The knot $D$ is $\varepsilon$-equivalent to the $(2,3)$-torus knot $T_{2,3}$ \cite[Lemma 6.12]{Homsmooth}. Moreover, if two knots are $\varepsilon$-equivalent, then so are their satellites \cite[Proposition 4]{Homsmooth}. Therefore, to understand $D$ and its satellites from the perspective of $\varepsilon$-equivalence, we may instead work with $T_{2,3}$ and its satellites.
The advantage of this is the knot Floer complex of $T_{2,3}$ is simpler to work with from a computational perspective. It has rank three, and is homologically thin, meaning that $\widehat{HFK}(T_{2,3})$ is supported on a single diagonal with respect to its bigrading.

Cables of homologically thin knots are studied by Petkova in \cite{Petkova}, where she describes $\widehat{HFK}(K_{p, pn+1})$ for any homologically thin knot $K$ in terms of the Alexander polynomial of $K$, $\tau(K)$, $p$, and $n$. The proof of her main result relies on bordered Floer homology, and the same techniques can be used to determine the $\Z$-filtered chain complex $\widehat{CFK}(K_{p, pn+1})$.

Since $T_{2,3}$ is homologically thin, we may use Theorem 1 of \cite{Petkova} to compute the $\Z$-filtered chain complex $\widehat{CFK}(T_{2,3; p, 1})$, from which we can determine certain information about $\CFKi(T_{2,3;p,1})$, which is $\varepsilon$-equivalent to $\CFKi(D_{p,1})$. More precisely, this information will be the invariants $a_1$ and $a_2$, which will determine the bounds on concordance genus necessary for Theorem \ref{thm:theknots}.

Towards this end, a useful tool is the well-known ``edge reduction'' procedure for filtered chain complexes over $\F$; see, for example, \cite[Section 2.6]{Levine}. That is, we may depict a filtered chain complex as a directed graph, where there is an arrow from $x_i$ to $x_j$ if $x_j$ appears with non-zero coefficient in $\partial x_i$. We label the arrow from $x_i$ to $x_j$ with the Alexander filtration difference between $x_i$ and $x_j$. If there is an arrow from $x_i$ to $x_j$ that preserves filtration, we may cancel it by deleting $x$ and $y$ from the graph, and for each $k$ and $\ell$ with edges
\begin{align*}
	x_k &\xrightarrow{a} x_j  \\
	x_i &\xrightarrow{b} x_\ell,
\end{align*}
we either add an arrow from $x_k$ to $x_\ell$ if one was not there previously, or delete the arrow from $x_k$ to $x_\ell$ if there was one. See Figure \ref{fig:edgereduction}. If we add an arrow from $x_k$ to $x_\ell$, then its filtration shift is $a+b$ where $a$ and $b$ where the filtration shifts of the arrows from $x_k$ to $x_j$ and from $x_i$ to $x_\ell$, respectively. This procedure corresponds to the following chain homotopy equivalence, consisting of a change of basis which yields an acyclic subcomplex: 
\begin{itemize}
	\item For each $x_k$ with an arrow to $x_j$, we replace $x_k$ with $x_k+x_i$.
	\item The basis element $x_j$ is replaced with $\partial x_i$.
	\item The subcomplex spanned by $x_i$ and $\partial x_i$ is acyclic.
\end{itemize}
We make use of this procedure in the following proposition.

\begin{figure}[htb!]
\begin{tikzpicture}
	\filldraw (0, 2) circle (1.5pt) node[] (au_2) {};
		\node [above] at (au_2) {$x_i$};
	\filldraw (3, 2) circle (1.5pt) node[] (b_1mu_1) {};
		\node [above] at (b_1mu_1) {$x_k$};
	\filldraw (3, 0) circle (1.5pt) node[] (b_2p-2mu_1) {};
		\node [right] at (b_2p-2mu_1) {$x_\ell$};
	\filldraw (0, 0) circle (1.5pt) node[] (au_3) {};
		\node [below] at (au_3) {$x_\ell$};
	\filldraw (-2, 0) circle (1.5pt) node[] (b_2p-2v_1) {};
		\node [left] at (b_2p-2v_1) {$x_j$};
	\filldraw (-2, 2) circle (1.5pt) node[] (b_1v_1) {};
		\node [above] at (b_1v_1) {$x_k$};
	\draw [very thick, ->] (au_2) -- (au_3) node[midway, right] {$b$};
	\draw [very thick, ->] (b_1mu_1) -- (b_2p-2mu_1) node[midway, right] {$a+b$};
	\draw [very thick, ->] (au_2) -- (b_2p-2v_1) node[midway, above] {$0$};
	\draw [very thick, ->] (b_1v_1) -- (b_2p-2v_1) node[midway, left] {$a$};
\end{tikzpicture}
\centering
\caption{An example of edge reduction. Left, before reduction; right, after.}
\label{fig:edgereduction}
\end{figure}
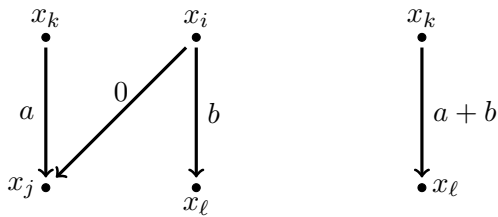

\begin{proposition}
\label{prop:cable}
The group $\widehat{HFK}(T_{2,3;p,1})$ has rank $6p-5$. The generators are listed in Table \ref{tab:cable}, and the non-zero higher differentials are
\begin{align*}
	\partial b_1 v_1 &= b_1\mu_1[p] \\
	\partial b_j v_1 &= b_{2p-j-1}v_1[p-j]  \qquad 2 \leq j \leq p-1 \\
	\partial b_jv_2 &= b_{j+1}\mu_1[1] \ \quad \qquad 1 \leq j \leq p-2 \\
	\partial b_{p-1}v_2 &= b_pv_2[1] \\
	\partial b_j \mu_2 &= b_{2p-j-1}\mu_2[p-j] \quad \quad 1 \leq j \leq p-1,
\end{align*}
where the brackets denote the drop in Alexander filtration, e.g., the Alexander filtration of $b_1 \mu_1$ is $p$ less than that of $b_1v_1$.
\end{proposition}

\begin{table}[htb!]
\vspace{5pt}
\begin{center}
\begin{tabular}{llllll}
\hline
&Generator \qquad \qquad & $(M, A)$ \qquad \qquad \qquad \qquad & $M+2p-2A$ &&\\\hline
&$au_1$ & $(0, p)$ & $0$ &\\
&$b_1v_1$  & $(-1, p-1)$ & $1$ &\\
&$b_1\mu_1$ & $(-2, -1)$ & $2p$ &\\
&$b_jv_2$ & $(-2j-1, -j)$ & $2p-1$ &$1 \leq j \leq p-2$&\\
&$b_{j+1}\mu_1$ \qquad \quad & $(-2j-2, -j-1)$  & $2p$ &$1 \leq j \leq p-2$&\\
&$b_{p-1}v_2 $ & $(-2p+1, -p+1)$ & $2p-1$ &&\\
&$b_pv_2$ & $(-2p, -p)$ & $2p$ & \\
&$b_jv_1$ & $(-1, -j+p)$ & $-1+2j$ &$2 \leq j \leq p-1$&\\
&$b_{2p-1-j}v_1$ & $(-2, 0)$ &  $2p-2$ &$2 \leq j \leq p-1$&\\
&$b_j\mu_2$ & $(0, -j+p)$ & $-2j$ &$1 \leq j \leq p-1$ &\\
&$b_{2p-1-j}\mu_2$ & $(-1, 0)$ & $2p-1$ &$1 \leq j \leq p-1$ &\\\hline
\end{tabular}
\end{center}
\caption{$\widehat{HFK}(T_{2,3;p,1})$}
\label{tab:cable}
\end{table}

\begin{proof}
We use \cite[Theorem 1]{Petkova} to determine $\widehat{HFK}(T_{2,3;p,1})$. We also need the higher differentials (i.e., those that do not preserve Alexander grading) in order to determine the values of $a_1$ and $a_2$, and so we repeat the calculation of $\widehat{HFK}(T_{2,3;p,1})$ below, keeping track of this additional data.

For background on bordered Floer homology, see \cite{LOT} or \cite[Section 2]{Homcables}. We prefer to work with $\Z$-filtered chain complex $\widehat{CFK}$ rather than the $\F[U]$-module $gCFK^-$, and so we use the basepoint conventions described in \cite[Remark 4.2]{Homcables}. In particular, the $\mathcal{A}_\infty$ relations on $\widehat{CFA}$ now each contribute a relation filtration shift, denoted with square brackets.

\begin{figure}[htb!]
\labellist
\small \hair 2pt
\pinlabel $w$ at 50 300
\pinlabel $z$ at 180 135
\pinlabel $0$ at 35 330
\pinlabel $1$ at 320 330
\pinlabel $2$ at 320 40
\pinlabel $3$ at 35 40
\pinlabel $a$ at 18 75
\pinlabel $b_1$ at 85 20
\pinlabel $\ldots$ at 108 20
\pinlabel $b_{p-1}$ at 140 20
\pinlabel $b_p$ at 230 20
\pinlabel $\ldots$ at 251 20
\pinlabel $b_{2p-2}$ at 280 20
\endlabellist
\centering
\includegraphics[scale=1]{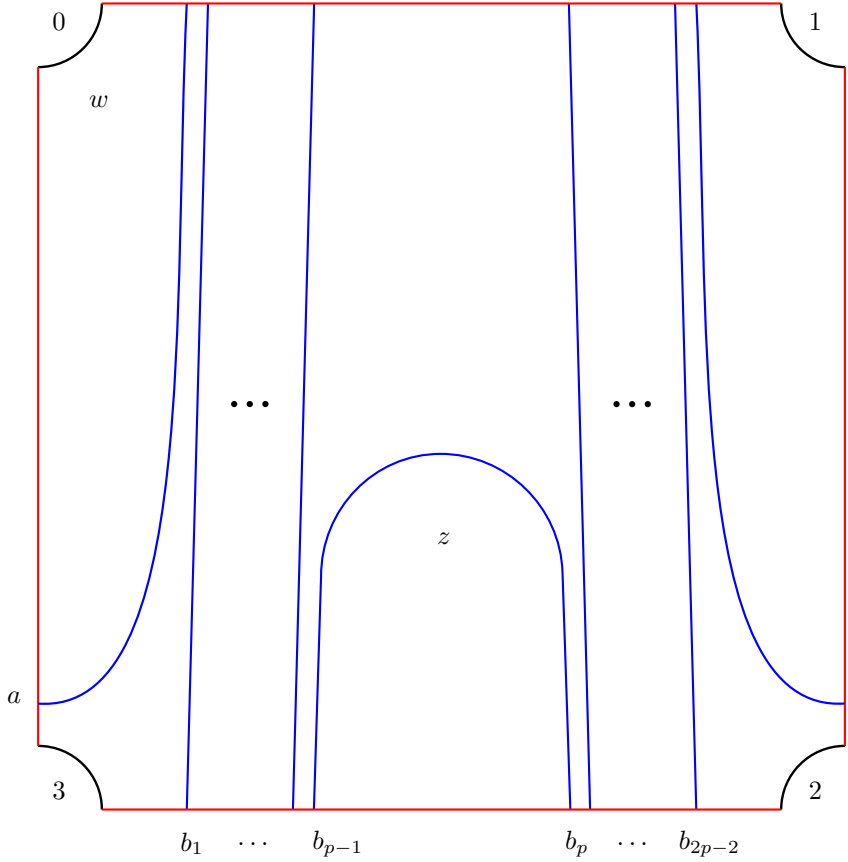}
\caption{A genus one bordered Heegaard diagram $\mathcal{H}(p,1)$ for the $(p,1)$-cable in the solid torus.}
\label{fig:cablepattern}
\end{figure}

We use the notation of \cite{Petkova}, which matches that of \cite{LOT}. Let $\widehat{CFA}(p, 1)$ denote the type $A$ structure associated to the diagram in Figure \ref{fig:cablepattern}. The diagram describes the $(p, 1)$-torus knot in the solid torus, where $p$ denotes the longitudinal winding. There are $2p-1$ generators, $a, b_1, b_2, \ldots, b_{2p-2}$, where $a \cdot \iota_0=a$, and $b_i \cdot \iota_1 = b_i$.
In \cite[Section 4.1]{Homcables}, the $\mathcal{A}_\infty$ relations on $\widehat{CFA}(p, 1)$ are determined to be
\begin{equation*}
\begin{array}{lll}
m_{3+i}(a, \rho_3, \overbrace{\rho_{23}, \ldots, \rho_{23}}^{i}, \rho_2)=a[pi+p], &&i \geq 0 \\
m_{4+i+j}(a, \rho_3, \overbrace{\rho_{23}, \ldots, \rho_{23}}^{i}, \rho_2, \overbrace{\rho_{12}, \ldots, \rho_{12}}^{j}, \rho_1)=b_{j+1}[pi+j+1], && 0\leq j \leq p-2\\
&& i \geq 0\\
m_{2+j}(a, \overbrace{\rho_{12}, \ldots, \rho_{12}}^{j}, \rho_1)=b_{2p-j-2}[0], &&0\leq j \leq p-2 \\
&&\\
m_1(b_j)=b_{2p-j-1}[p-j], &&1\leq j \leq p-1 \\
m_{3+i}(b_j, \rho_2, \overbrace{\rho_{12}, \ldots, \rho_{12}}^{i}, \rho_1)=b_{j+i+1}[i+1], &&1\leq j \leq p-2\\
&& 0\leq i \leq p-j-2 \\
m_{3+i}(b_j, \rho_2, \overbrace{\rho_{12}, \ldots, \rho_{12}}^{i}, \rho_1)=b_{j-i-1}[0], &&p+1\leq j \leq 2p-2\\
&& 0 \leq i \leq j-p-1.\\
\end{array}
\end{equation*}

Let $Y$ denote the $0$-framed complement of the right-handed trefoil. By \cite[Theorem A.11]{LOT}, the type $D$ structure $\widehat{CFD}(Y)$ is as shown in Figure \ref{fig:CFDRHT}, where the generators $u_i$ are in the idempotent $\iota_0$, i.e., $\iota_0 \cdot u_i = u_i$, and the remaining generators are in the idempotent $\iota_1$.

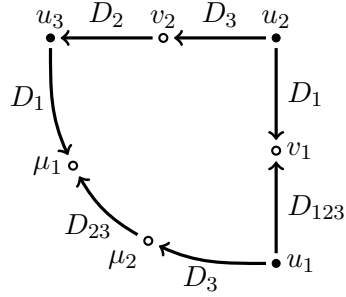
\begin{figure}[htb!]
\begin{tikzpicture}
	\useasboundingbox (-0.5, -0.5) rectangle (3.5, 3.5);
	
        \filldraw (0, 3) circle (1.5pt) node[] (a) {};
        \draw [thick] (1.5, 3) circle (1.5pt) node[] (ab) {};
        \filldraw (3, 3) circle (1.5pt) node[] (b) {};
        \draw [thick] (3, 1.5) circle (1.5pt) node[] (bc) {};        
        \filldraw (3, 0) circle (1.5pt) node[]  (c) {};
        \draw [thick] (0.3, 1.3) circle (1.5pt) node[] (m1) {};
        \draw [thick] (1.3, 0.3) circle (1.5pt) node[] (m2) {};      

         	\draw [very thick, <-] (a) -- (ab) node[midway, above] {$D_{2}$};
	\draw [very thick, <-] (ab) -- (b) node[midway, above] {$D_{3}$};
         	\draw [very thick, ->] (b) -- (bc) node[midway, right] {$D_{1}$};
	\draw [very thick, ->] (c) -- (bc) node[midway, right] {$D_{123}$};
	\draw [very thick, ->] (a) to [out=270, in=115] (m1);
	\draw [very thick, ->] (m2) to [out=150, in=300] (m1);
	\draw [very thick, ->] (c) to [out=180, in=335] (m2);
	\draw (-0.3, 2.2) node {$D_{1}$};
	\draw (0.5, 0.5) node {$D_{23}$};
	\draw (2, -0.2) node {$D_{3}$};
	
	\node [above] at (a) {$u_3$};
	\node [above] at (b) {$u_2$};
	\node [right] at (c) {$u_1$};
	\node [above] at (ab) {$v_2$};
	\node [right] at (bc) {$v_1$};
	\node [left] at (m1) {$\mu_1$};
	\node [below left] at (m2) {$\mu_2$};
\end{tikzpicture}
\centering
\label{fig:CFDRHT}
\caption{$\widehat{CFD}(Y)$, where $Y$ is the $0$-framed complement of the right-handed trefoil.}
\label{}
\end{figure}

The generators and differentials of $\widehat{CFA}(p, 1) \boxtimes \widehat{CFD}(Y)$ are
\begin{equation*}
\begin{array}{ll}
	\partial (au_1) = 0 &\\
	\partial (au_2) = au_3[p] + b_1\mu_1[1] + b_{2p-2}v_1[0] \quad &\\
	\partial (au_3) = b_{2p-2}\mu_1[0] &\\
	\partial (b_jv_1) = b_{2p-j-1}v_1[p-j], & 1 \leq j \leq p-1 \\
	\partial (b_jv_2) = b_{2p-j-1}v_2[p-j] + b_{j+1}\mu_1[1], & 1 \leq j \leq p-2 \\
	\partial (b_{p-1}v_2 = b_pv_2[1] &\\
	\partial (b_j \mu_1) = b_{2p-j-1}\mu_1[p-j], & 1\leq j \leq p-1 \\
	\partial (b_j \mu_2) = b_{2p-j-1}\mu_2[p-j], & 1 \leq j \leq p-1 \\
	\partial (b_pv_1) = 0 &\\
	\partial (b_pv_2) = 0 &\\
	\partial (b_p\mu_1) = 0 &\\
	\partial (b_p\mu_2) = 0 &\\
	\partial (b_jv_1) = 0, & p+1 \leq j \leq 2p-2 \\
	\partial (b_jv_2) = b_{j-1}\mu_1[0], & p+1 \leq j \leq 2p-2 \\
	\partial (b_j\mu_1) = 0, & p+1 \leq j \leq 2p-2 \\
	\partial (b_j\mu_2) = 0, & p+1 \leq j \leq 2p-2. \\
\end{array}
\end{equation*}
The change in Alexander filtration, denoted in square brackets, can be determined from the relative Alexander filtration shifts in the $\mathcal{A}_\infty$ relations on $\widehat{CFA}(p, 1)$.

There is a summand of $\widehat{CFK}(T_{2,3;p,1})$ consisting of the generators
\[ au_2, \ \ au_3, \ \ b_1v_1, \ \ b_1\mu_1, \ \ b_{2p-2}v_1, \ \ b_{2p-2}\mu_1 \]
with the nonzero differentials
\begin{align*}
\partial (au_2) &= au_3[p] + b_1\mu_1[1] + b_{2p-2}v_1[0] \\
\partial (b_1v_1) &= b_{2p-2}v_1[p-1] \\
\partial (b_1\mu_1) &= b_{2p-2}\mu_1[p-1] \\
\partial (au_3) &= b_{2p-2}\mu_1[0].
\end{align*}
See Figure \ref{subfig:summand1a}. We cancel the edge the edge between $au_2$ and $b_{2p-2}v_1$, and the edge between $au_3$ and $b_{2p-2}\mu_1$, which introduces an edge between $b_1v_1$ and $b_1\mu_1$. The summand now consists of
\[ b_1v_1, \ \ b_1\mu_1 \]
with the differential
\[ \partial (b_1v_1) = b_1\mu_1[p].\]
See Figure \ref{subfig:summand1}. Similarly, when $p\geq 3$, there is a summand of $\widehat{CFK}(T_{2,3;p,1})$ consisting of the generators
\[ b_jv_2, \ \ b_{2p-j-1}v_2, \ \ b_{j+1}\mu_1, \ \ b_{2p-j-2}\mu_1 \]
 for $1 \leq j \leq p-2$, with the following nonzero differentials
\begin{align*}
\partial (b_jv_2) &= b_{2p-j-1}v_2[p-1]+b_{j+1}\mu_1[1] \\
\partial (b_{2p-j-1}v_2) &= b_{2p-j-2}\mu_1[0]\\
\partial (b_{j+1}\mu_1) &= b_{2p-j-2}\mu_1[p-2].
\end{align*}
After canceling the edge between $b_{2p-j-1}v_2$ and $b_{2p-j-2}\mu_1$, we reduce the summand to
\[ b_jv_2, \ \ b_{j+1}\mu_1 \]
with the nonzero differential
\[ \partial (b_jv_2) = b_{j+1}\mu_1[1].\]
See Figure \ref{fig:summand2}. The remaining summands of $\widehat{CFK}(T_{2,3;p,1})$ are shown in Figure \ref{fig:summand3}.

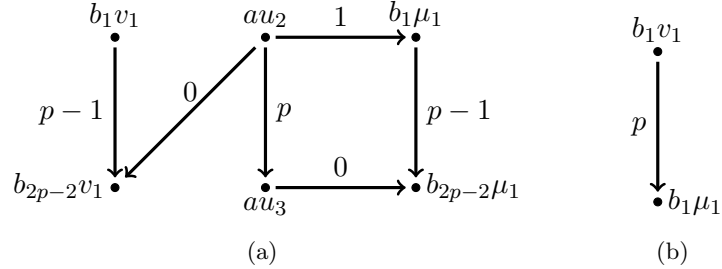
\begin{figure}[htb!]
\subfigure[]{
\begin{tikzpicture}
	\filldraw (0, 2) circle (1.5pt) node[] (au_2) {};
		\node [above] at (au_2) {$au_2$};
	\filldraw (2, 2) circle (1.5pt) node[] (b_1mu_1) {};
		\node [above] at (b_1mu_1) {$b_1\mu_1$};
	\filldraw (2, 0) circle (1.5pt) node[] (b_2p-2mu_1) {};
		\node [right] at (b_2p-2mu_1) {$b_{2p-2}\mu_1$};
	\filldraw (0, 0) circle (1.5pt) node[] (au_3) {};
		\node [below] at (au_3) {$au_3$};
	\filldraw (-2, 0) circle (1.5pt) node[] (b_2p-2v_1) {};
		\node [left] at (b_2p-2v_1) {$b_{2p-2}v_1$};
	\filldraw (-2, 2) circle (1.5pt) node[] (b_1v_1) {};
		\node [above] at (b_1v_1) {$b_1v_1$};
	\draw [very thick, ->] (au_2) -- (b_1mu_1) node[midway, above] {$1$};
	\draw [very thick, ->] (au_2) -- (au_3) node[midway, right] {$p$};
	\draw [very thick, ->] (au_3) -- (b_2p-2mu_1) node[midway, above] {$0$};
	\draw [very thick, ->] (b_1mu_1) -- (b_2p-2mu_1) node[midway, right] {$p-1$};
	\draw [very thick, ->] (au_2) -- (b_2p-2v_1) node[midway, above] {$0$};
	\draw [very thick, ->] (b_1v_1) -- (b_2p-2v_1) node[midway, left] {$p-1$};
\end{tikzpicture}
\label{subfig:summand1a}
}
\hspace{20pt}
\subfigure[]{
\begin{tikzpicture}
	\filldraw (0, 2) circle (1.5pt) node[] (b_1v_1) {};
		\node [above] at (b_1v_1) {$b_1v_1$};
	\filldraw (0, 0) circle (1.5pt) node[] (b_1mu_1) {};
		\node [right] at (b_1mu_1) {$b_1\mu_1$};
	\draw [very thick, ->] (b_1v_1) -- (b_1mu_1) node[midway, left] {$p$};
\end{tikzpicture}
\label{subfig:summand1}
}
\centering
\caption{A summand of $\widehat{CFK}(T_{2,3;p,1})$. Left, before any simplifications. Right, after canceling the differential from $au_3$ to $b_{2p-2}\mu_1$ and the differential from $au_2$ to $b_{2p-2}v_1$. The labels on the arrows indicate the change in filtration.}
\label{fig:summand1}
\end{figure}

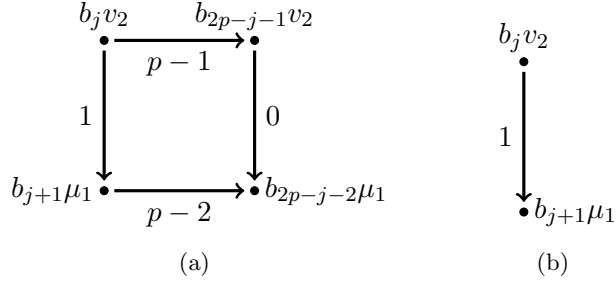
\begin{figure}[htb!]
\subfigure[]{
\begin{tikzpicture}
	\filldraw (0, 2) circle (1.5pt) node[] (b_jv_2) {};
		\node [above] at (b_jv_2) {$b_jv_2$};
	\filldraw (2, 2) circle (1.5pt) node[] (b_2p-j-1v_2) {};
		\node [above] at (b_2p-j-1v_2) {$b_{2p-j-1}v_2$};
	\filldraw (2, 0) circle (1.5pt) node[] (b_2p-j-2mu_1) {};
		\node [right] at (b_2p-j-2mu_1) {$b_{2p-j-2}\mu_1$};
	\filldraw (0, 0) circle (1.5pt) node[] (b_j+1mu_1) {};
		\node [left] at (b_j+1mu_1) {$b_{j+1}\mu_1$};
	\draw [very thick, ->] (b_jv_2) -- (b_2p-j-1v_2) node[midway, below] {$p-1$};
	\draw [very thick, ->] (b_jv_2) -- (b_j+1mu_1) node[midway, left] {$1$};
	\draw [very thick, ->] (b_2p-j-1v_2) -- (b_2p-2mu_1) node[midway, right] {$0$};
	\draw [very thick, ->] (b_j+1mu_1) -- (b_2p-2mu_1) node[midway, below] {$p-2$};
\end{tikzpicture}
}
\hspace{20pt}
\subfigure[]{
\begin{tikzpicture}
	\filldraw (0, 2) circle (1.5pt) node[] (b_jv_2) {};
		\node [above] at (b_jv_2) {$b_jv_2$};
	\filldraw (0, 0) circle (1.5pt) node[] (b_j+1mu_1) {};
		\node [right] at (b_j+1mu_1) {$b_{j+1}\mu_1$};
	\draw [very thick, ->] (b_jv_2) -- (b_j+1mu_1) node[midway, left] {$1$};
\end{tikzpicture}
\label{subfig:summand2}
}
\centering
\caption{A summand of $\widehat{CFK}(T_{2,3;p,1})$, where $1 \leq j \leq p-2$. Left, before any simplifications; right, after.}
\label{fig:summand2}
\end{figure}

\begin{figure}[htb!]
\begin{tikzpicture}
	\filldraw (0, 2) circle (1.5pt) node[] (b_iv_1) {};
		\node [above] at (b_iv_1) {$b_iv_1$};
	\filldraw (0, 0) circle (1.5pt) node[] (b_2p-i-1v_1) {};
		\node [below] at (b_2p-i-1v_1) {$b_{2p-i-1}v_1$};
	\draw [very thick, ->] (b_iv_1) -- (b_2p-i-1v_1) node[midway, left] {$p-i$};
	
	\filldraw (4, 2) circle (1.5pt) node[] (b_jmu_2) {};
		\node [above] at (b_jmu_2) {$b_j\mu_2$};
	\filldraw (4, 0) circle (1.5pt) node[] (b_2p-j-1mu_2) {};
		\node [below] at (b_2p-j-1mu_2) {$b_{2p-j-1}\mu_2$};
	\draw [very thick, ->] (b_jmu_2) -- (b_2p-j-1mu_2) node[midway, left] {$p-j$};
	
	\filldraw (8, 2) circle (1.5pt) node[] (b_p-1v_2) {};
		\node [above] at (b_p-1v_2) {$b_{p-1}v_2$};
	\filldraw (8, 0) circle (1.5pt) node[] (b_pv_2) {};
		\node [below] at (b_pv_2) {$b_pv_2$};
	\draw [very thick, ->] (b_p-1v_2) -- (b_pv_2) node[midway, left] {$1$};
\end{tikzpicture}
\centering
\caption{The remaining summands of $\widehat{CFK}(T_{2,3;p,1})$, where $2 \leq i \leq p-1$ and $1 \leq j \leq p-1$.}
\label{fig:summand3}
\end{figure}
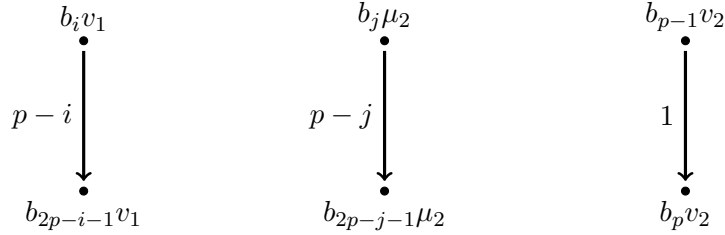

After applying the edge reduction procedure, the nonzero higher differentials on $\widehat{HFK}(T_{2,3;p,1})$ are
\begin{align*}
\partial (b_1v_1) &= b_1\mu_1[p] \\
\partial (b_iv_2)&=  b_{i+1}\mu_1[1], \quad \qquad 1 \leq i \leq p-2 \\
\partial (b_{p-1}v_2) &= b_pv_2[1] \\
\partial (b_iv_1) &= b_{2p-1-i}v_1[p-i], \qquad 2 \leq i \leq p-1 \\
\partial (b_{i}\mu_2 ) &= b_{2p-1-i}\mu_2[p-i], \qquad 1 \leq i \leq p-1,
\end{align*}
as depicted in Figures \ref{subfig:summand1}, \ref{subfig:summand2}, and \ref{fig:summand3}. We determine the gradings in Table \ref{tab:cable} using \cite[Theorem 1]{Petkova}. Due to our choice of basepoint conventions, our gradings differ from those in \cite{Petkova} in the following ways: our Alexander grading $A$ is the negative of Petkova's, and our Maslov grading $M$ is Petkova's $N$. This completes the proof of the proposition.
\end{proof}

The basis for $\widehat{HFK}$ in the above proposition has a particularly simple form. In the language of \cite[Definition 11.25]{LOT}, it is \emph{simplified}; that is, there is at most one arrow starting or ending at each basis element.

\begin{figure}[htb!]
\centering
\vspace{10pt}
\begin{tikzpicture}
	\draw[step=1, black!30!white, very thin] (-3.9, -6.9) grid (3.9, 0.9);
	\begin{scope}[thin, darkgray]
		\draw [<->] (-4, 0) -- (4, 0) node[right](xaxis){$A$};
		\draw [<->] (0, -7) -- (0, 1) node[above](yaxis){$M$};
	\end{scope}
	\filldraw (3, 0) circle (2pt) node[] (au_1){};
	\filldraw (2, -1) circle (2pt) node[] (b_1v_1){};
	\filldraw (-1,-2) circle (2pt) node[] (b_1mu_1){};
	\filldraw (-1, -3) circle (2pt) node[] (b_1v_2){};
	\filldraw (-2, -4) circle (2pt) node[] (b_2mu_1){};
	\filldraw (-2, -5) circle (2pt) node[] (b_2v_2){};
	\filldraw (-3, -6) circle (2pt) node[] (b_3v_2){};
	\filldraw (1, -1) circle (2pt) node[] (b_2v_1){};
	\filldraw (0, -2) circle (2pt) node[] (b_3v_1){};
	\filldraw (2, 0) circle (2pt) node[] (b_1mu_2){};
	\filldraw (1, 0) circle (2pt) node[] (b_2mu_2){};
	\filldraw (0.1, -1.1) circle (2pt) node[] (b_4mu_2){};
	\filldraw (-0.1, -0.9) circle (2pt) node[] (b_3mu_2){};

	\node [above] at (au_1) {$au_1$};
	\node [right] at (b_1v_1) {$b_1v_1$};
	\node [left] at (b_1mu_1) {$b_1\mu_1$};
	\node [right] at (b_1v_2) {$b_1v_2$};
	\node [left] at (b_2mu_1) {$b_2\mu_1$};
	\node [left] at (b_2v_2) {$b_2v_2$};
	\node [left] at (b_3v_2) {$b_3v_2$};
	\node [above right] at (b_2v_1) {$b_2v_1$};
	\node [right] at (b_3v_1) {$b_3v_1$};
	\node [above] at (b_1mu_2) {$b_1\mu_2$};
	\node [left] at (b_2mu_2) {$b_2\mu_2$};
	\node [left] at (b_4mu_2) {$b_4\mu_2$};
	\node [above left] at (b_3mu_2) {$b_3\mu_2$};

	\draw [->] (b_1v_1) -- (b_1mu_1);
	\draw [->] (b_1v_2) -- (b_2mu_1);
	\draw [->] (b_2v_2) -- (b_3v_2);
	\draw [->] (b_2v_1) -- (b_3v_1);
	\draw [->] (b_1mu_2) -- (b_4mu_2);	
	\draw [->] (b_2mu_2) -- (b_3mu_2);	
\end{tikzpicture}
\vspace{10pt}
\caption{$\widehat{CFK}(T_{2,3;3,1})$, where the horizontal axis represents the Alexander grading and the vertical axis represents the homological, or Maslov, grading.}
\label{fig:}
\end{figure}
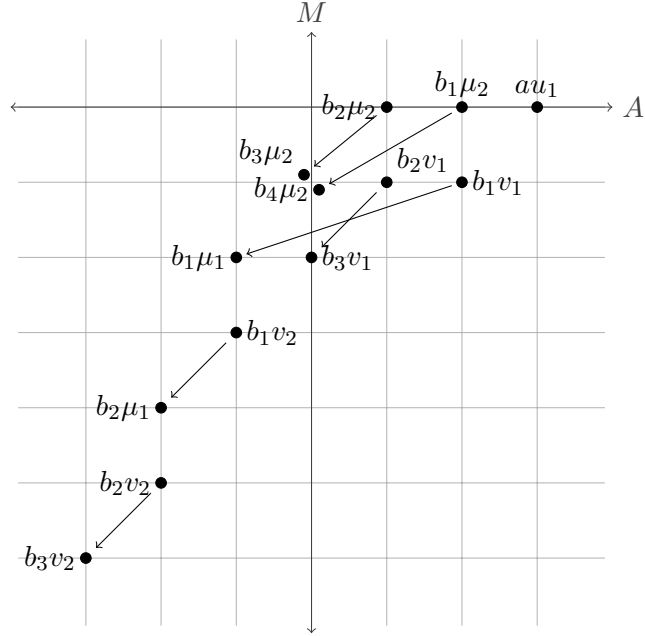

\begin{lemma}
\label{lem:a1a21p}
Let $D$ denote the (positive, untwisted) Whitehead double of the right-handed trefoil, and $D_{p,1}$ its $(p,1)$-cable, $p>1$. Then $a_1(D_{p,1})=1$ and $a_2(D_{p,1})=p$.
\end{lemma}

\begin{proof}
The knot $T_{2,3;p,1}$ is $\varepsilon$-equivalent to $D_{p,1}$, so we will study $\CFKi(T_{2,3;p,1})$ instead of $\CFKi(D_{p,1})$.

By \cite[Theorem 2]{Homcables}, we know that $\varepsilon(T_{2,3;p,1})=1$. By Proposition \ref{prop:cable}, we know that $au_1$ is a generator of the total homology $H_*(\widehat{CFK}(T_{2,3;p,1}))$.
We will now find a basis satisfying the conditions in Lemma \ref{lem:basis}, and in doing so, will determine the values of $a_1(T_{2,3;p,1})$ and $a_2(T_{2,3;p,1})$. In order to accomplish this, we will need to find an element whose horizontal boundary in $\CFKi(T_{2,3;p,1})$ is $au_1$.

We will view the $\Z \oplus \Z$-filtered chain complex $\CFKi(K)$ in the $(i,j)$-plane. The complex $\CFKi(K)$ is filtered chain homotopic to a complex generated over $\F[U, U^{-1}]$ by $\widehat{HFK}(K)$, and thus the complex $\widehat{HFK}(K)$ can be viewed as the subquotient complex of $\CFKi(K)$ consisting of elements with $i$-coordinate equal to zero. See Figure \ref{subfig:vertical}. We place a generator $x$ at the lattice point $(0, A(x))$, where $A(x)$ denotes the Alexander grading of $x$. For example, the generator $b_1v_1$ has coordinates $(0, p-1)$. Multiplication by $U$ decreases both the $i$- and $j$-coordinates by one. The Maslov grading is suppressed from the picture, although we will still keep track of it, and recall that an element $U^n \cdot x$ has $(i,j)$-coordinates $(-n, A(x)-n)$, and Maslov grading $M(x)-2n$..  

We would like to find an element with $j$-coordinate equal to $\tau(T_{2,3;p,1})$ whose horizontal boundary is equal to $au_1$. In particular, we would like to find an element with $j$-coordinate equal to $p$, $i$-coordinate greater than zero, and Maslov grading one, which is one more than the Maslov grading of $au_1$. To find the elements with $j$-coordinate equal to $p$, we view the appropriate $U$-translates of elements in $\widehat{HFK}(K)$. More specifically, given a generator $x$ of $\widehat{HFK}(K)$, the translate $U^{A(x)-p} \cdot x$ will be in the $p^{\textup{th}}$-row, with
\[ A(U^{A(x)-p} \cdot x)=p \qquad \textup{and} \qquad M(U^{A(x)-p} \cdot x)=M(x)+2p-2A(x).\]
See Figure \ref{subfig:horizontal}. By considering the gradings in the third column of Table \ref{tab:cable}, which are the Maslov gradings of the elements in the  $p^{\textup{th}}$-row, we see that the only element in that row with Maslov grading one is $U^{-1} \cdot b_1v_1$.

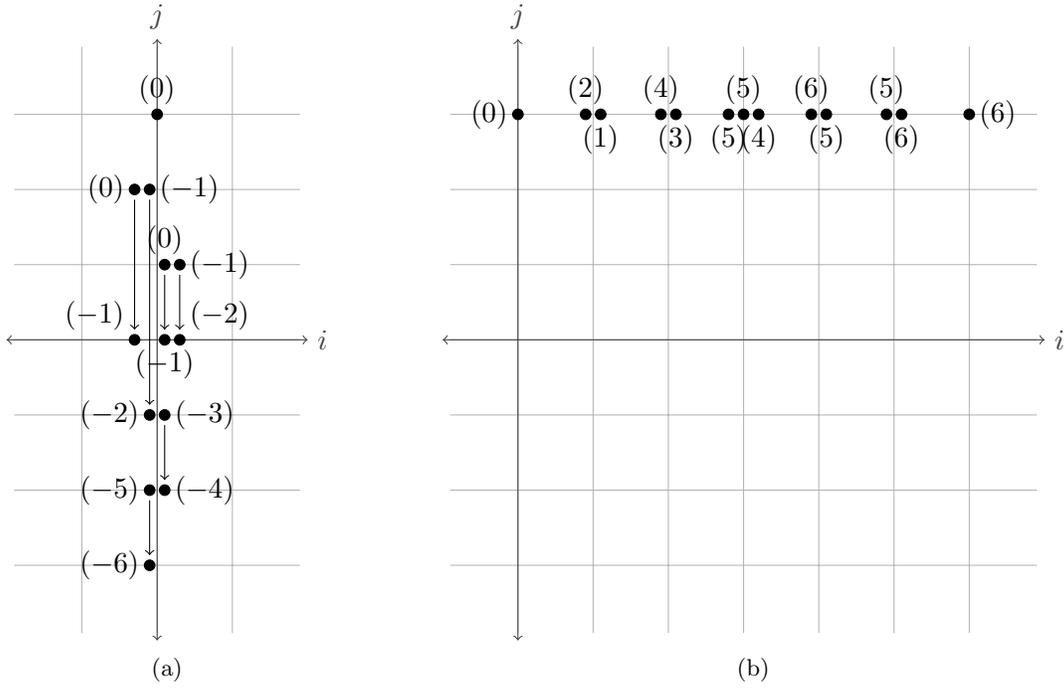
\begin{figure}[htb!]
\centering
\vspace{10pt}
\subfigure[]{
\begin{tikzpicture}
	\draw[step=1, black!30!white, very thin] (-1.9, -3.9) grid (1.9, 3.9);
	\begin{scope}[thin, darkgray]
		\draw [<->] (-2, 0) -- (2, 0) node[right](xaxis){$i$};
		\draw [<->] (0, -4) -- (0, 4) node[above](yaxis){$j$};
	\end{scope}
	\filldraw (0, 3) circle (2pt) node[] (au_1){};
		\node [above] at (au_1) {$(0)$};
		
	\filldraw (-0.1, 2) circle (2pt) node[] (b_1v_1){};
		\node [right] at (b_1v_1) {$(-1)$}; 
	\filldraw (-0.1, -1) circle (2pt) node[] (b_1mu_1){};
		\node [left] at (b_1mu_1) {$(-2)$};
	\draw [->] (b_1v_1) -- (b_1mu_1);
		
	\filldraw (-0.3, 2) circle (2pt) node[] (b_1mu_2){};
		\node [left] at (b_1mu_2) {$(0)$};
	\filldraw (-0.3, 0) circle (2pt) node[] (b_4mu_2){};
		\node [above left] at (b_4mu_2) {$(-1)$};
	\draw [->] (b_1mu_2) -- (b_4mu_2);	
	
	\filldraw (0.1, 1) circle (2pt) node[] (b_2mu_2){};
		\node [above] at (b_2mu_2) {$(0)$};
	\filldraw (0.1, 0) circle (2pt) node[] (b_3mu_2){};
		\node [below] at (b_3mu_2) {$(-1)$};
	\draw [->] (b_2mu_2) -- (b_3mu_2);
	
	\filldraw (0.3, 1) circle (2pt) node[] (b_2v_1){};
		\node [right] at (b_2v_1) {$(-1)$};
	\filldraw (0.3, 0) circle (2pt) node[] (b_3v_1){};
		\node [above right] at (b_3v_1) {$(-2)$};
	\draw [->] (b_2v_1) -- (b_3v_1);

	\filldraw (0.1, -1) circle (2pt) node[] (b_1v_2){};
		\node [right] at (b_1v_2) {$(-3)$};
	\filldraw (0.1, -2) circle (2pt) node[] (b_2mu_1){};
		\node [right] at (b_2mu_1) {$(-4)$};
	\draw [->] (b_1v_2) -- (b_2mu_1);
	
	\filldraw (-0.1, -2) circle (2pt) node[] (b_2v_2){};
		\node [left] at (b_2v_2) {$(-5)$};
	\filldraw (-0.1, -3) circle (2pt) node[] (b_3v_2){};
		\node [left] at (b_3v_2) {$(-6)$};
	\draw [->] (b_2v_2) -- (b_3v_2);

\end{tikzpicture}
\label{subfig:vertical}
}
\hspace{25pt}
\subfigure[]{
\begin{tikzpicture}
	\draw[step=1, black!30!white, very thin] (-.9, -3.9) grid (6.9, 3.9);
	\begin{scope}[thin, darkgray]
		\draw [<->] (-1, 0) -- (7, 0) node[right](xaxis){$i$};
		\draw [<->] (0, -4) -- (0, 4) node[above](yaxis){$j$};
	\end{scope}
	\filldraw (0, 3) circle (2pt) node[] (a){};
		\node [left] at (a) {$(0)$};
	\filldraw (0.9, 3) circle (2pt) node[] (b){};
		\node [above] at (b) {$(2)$};
	\filldraw (1.1, 3) circle (2pt) node[] (c){};
		\node [below] at (c) {$(1)$};
	\filldraw (1.9, 3) circle (2pt) node[] (d){};
		\node [above] at (d) {$(4)$};
	\filldraw (2.1, 3) circle (2pt) node[] (e){};
		\node [below] at (e) {$(3)$};
	\filldraw (2.8, 3) circle (2pt) node[] (f){};
		\node [below] at (f) {$(5)$};
	\filldraw (3, 3) circle (2pt) node[] (g){};
		\node [above] at (g) {$(5)$};
	\filldraw (3.2, 3) circle (2pt) node[] (h){};
		\node [below] at (h) {$(4)$};
	\filldraw (3.9, 3) circle (2pt) node[] (i){};
		\node [above] at (i) {$(6)$};
	\filldraw (4.1, 3) circle (2pt) node[] (j){};
		\node [below] at (j) {$(5)$};
	\filldraw (4.9, 3) circle (2pt) node[] (k){};
		\node [above] at (k) {$(5)$};
	\filldraw (5.1, 3) circle (2pt) node[] (l){};
		\node [below] at (l) {$(6)$};
	\filldraw (6, 3) circle (2pt) node[] (m){};
		\node [right] at (m) {$(6)$};

\end{tikzpicture}
\label{subfig:horizontal}
}
\vspace{10pt}
\caption{Left, the complex $\widehat{HFK}(T_{2,3;3,1})$, in the $(i, j)$-plane, with the vertical higher differentials. Right, the $U$-translates of the complex $\widehat{HFK}(T_{2,3;3,1})$ to the $j=3$ row. The numbers in parentheses indicate the Maslov gradings of the generators.}
\label{fig:}
\end{figure}

Thus, by grading considerations, we have concluded that $au_1$ is the horizontal boundary of $U^{-1}\cdot b_1v_1$. The vertical boundary of $U^{-1} \cdot b_1v_1$ is $U^{-1} \cdot b_1\mu_1$, and 
\[ A(U^{-1} \cdot b_1v_1)  = A(U^{-1} \cdot b_1\mu_1) +p. \]
It follows that
\[a_1(T_{2,3;p,1})=1 \qquad \textup{and} \qquad a_2(T_{2,3;p,1})=p, \]
and since $T_{2,3;p,1}$ and $D_{p,1}$ are $\varepsilon$-equivalent, the result follows.
\end{proof}

\begin{figure}[htb!]
\centering
\vspace{10pt}
\begin{tikzpicture}
	\begin{scope}[thin, gray]
		\draw [<->] (-1, 0) -- (3, 0);
		\draw [<->] (0, -1) -- (0, 5);
	\end{scope}
	\filldraw (0, 3) circle (2pt) node[] (x_0){};
	\filldraw (1, 3) circle (2pt) node[] (x_1){};
	\filldraw (1, 0) circle (2pt) node[] (x_2){};
	\draw [very thick, <-] (x_0) -- (x_1);
	\draw [very thick, <-] (x_2) -- (x_1);
	\node [left] at (x_0) {$au_1$};
	\node [right] at (x_1) {$U^{-1} b_1v_1$};
	\node [below] at (x_2) {$U^{-1} b_1\mu_1$};
\end{tikzpicture}
\vspace{10pt}
\caption{The elements of interest in the proof of Lemma \ref{lem:a1a21p}.}
\label{fig:}
\end{figure}
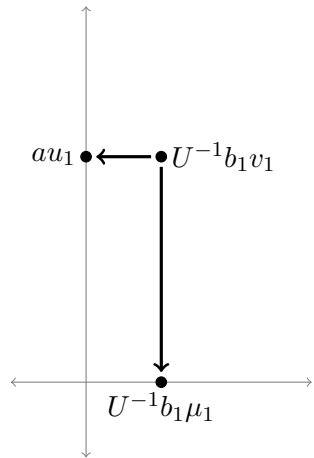

We are now ready to prove Theorem \ref{thm:theknots}, giving an infinite family of topologically slice knots with $4$-ball genus one and arbitrarily large concordance genus.

\begin{proof}[Proof of Theorem \ref{thm:theknots}]
By Lemma \ref{lem:a1a21p},
\[ a_1(D_{p,1})=1 \qquad \textup{and} \qquad a_2(D_{p,1})=p. \]
In the proof of \cite[Lemma 6.4]{Homsmooth}, it is shown that given knots $J$ and $K$, if $a_1(J)=a_1(K)$ and $a_2(J)>a_2(K)$, then
\[ a_1(J \# -K)=a_1(J) \qquad \textup{and} \qquad a_2(J \# -K)=a_2(J). \]
In particular,
\[ a_1(D_{p,1} \# -D_{p-1,1})=1 \qquad \textup{and} \qquad a_2(D_{p,1} \# -D_{p-1,1})=p. \]

In the beginning of this section, it was observed that the knots $D_{p,1} \# -D_{p-1,1}$ are topologically slice, and in Lemma \ref{lem:4genus1}, we saw that $g_4(D_{p,1} \# -D_{p-1,1})=1$.

By Theorem \ref{thm:gamma} and Lemma \ref{lem:gamma}, we see that
\begin{align*}
g_c(D_{p,1} \# -D_{p-1,1}) &\geq |\tau(D_{p,1} \# -D_{p-1,1})-a_1(D_{p,1} \# -D_{p-1,1})-a_2(D_{p,1} \# -D_{p-1,1})| \\
&= |1-1-p| \\
&= p,
\end{align*}
completing the proof of the theorem.
\end{proof}

\bibliographystyle{amsalpha}
\bibliography{mybib}

\providecommand{\bysame}{\leavevmode\hbox to3em{\hrulefill}\thinspace}
\providecommand{\MR}{\relax\ifhmode\unskip\space\fi MR }
\providecommand{\MRhref}[2]{%
  \href{http://www.ams.org/mathscinet-getitem?mr=#1}{#2}
}
\providecommand{\href}[2]{#2}
\begin{thebibliography}{LOT08}

\bibitem[Fre82]{Freedman}
Michael~Hartley Freedman, \emph{The topology of four-dimensional manifolds}, J.
  Differential Geom. \textbf{17} (1982), no.~3, 357--453.

\bibitem[Gor78]{Gordonproblems}
C.~McA. Gordon, \emph{Problems}, Knot theory ({P}roc. {S}em.,
  {P}lans-sur-{B}ex, 1977), Lecture Notes in Math., vol. 685, Springer, Berlin,
  1978, pp.~309--311.

\bibitem[Hed07]{HeddenWhitehead}
Matthew Hedden, \emph{Knot {F}loer homology of {W}hitehead doubles}, Geom.
  Topol. \textbf{11} (2007), 2277--2338.

\bibitem[Hed09]{HeddencablingII}
\bysame, \emph{On knot {F}loer homology and cabling {II}}, Int. Math. Res. Not.
  IMRN (2009), no.~12, 2248--2274.

\bibitem[Hom11]{Homsmooth}
Jennifer Hom, \emph{The knot {F}loer complex and the smooth concordance group},
  preprint (2011), \ to appear in Comment. Math. Helv., available at
  arXiv:1111.6635v1.

\bibitem[Hom12]{Homcables}
\bysame, \emph{Bordered {H}eegaard {F}loer homology and the tau-invariant of
  cable knots}, preprint (2012), to appear in J. Topology, available at
  arXiv:1202.1463v2.

\bibitem[KM93]{KronMrowka}
P.~B. Kronheimer and T.~S. Mrowka, \emph{Gauge theory for embedded surfaces.
  {I}}, Topology \textbf{32} (1993), no.~4, 773--826. \MR{1241873 (94k:57048)}

\bibitem[Lev10]{Levine}
Adam Levine, \emph{Knot doubling operators and bordered {H}eegaard {F}loer
  homology}, preprint (2010), \ arXiv:1008.3349v1.

\bibitem[Liv04]{Livingstonconcordancegenus}
Charles Livingston, \emph{The concordance genus of knots}, Algebr. Geom. Topol.
  \textbf{4} (2004), 1--22.

\bibitem[LOT08]{LOT}
Robert Lipshitz, Peter Ozsv\'ath, and Dylan Thurston, \emph{Bordered {H}eegaard
  {F}loer homology: {I}nvariance and pairing}, preprint (2008), \
  arXiv:0810.0687v4.

\bibitem[Nak81]{Nakanishi}
Yasutaka Nakanishi, \emph{A note on unknotting number}, Math. Sem. Notes Kobe
  Univ. \textbf{9} (1981), no.~1, 99--108. \MR{634000 (83d:57005)}

\bibitem[OS03]{OS4ball}
Peter Ozsv{\'a}th and Zolt{\'a}n Szab{\'o}, \emph{Knot {F}loer homology and the
  four-ball genus}, Geom. Topol. \textbf{7} (2003), 615--639.

\bibitem[OS04a]{OSgenusbounds}
\bysame, \emph{Holomorphic disks and genus bounds}, Geom. Topol. \textbf{8}
  (2004), 311--334.

\bibitem[OS04b]{OSknots}
\bysame, \emph{Holomorphic disks and knot invariants}, Adv. Math. \textbf{186}
  (2004), no.~1, 58--116.

\bibitem[OS06]{OSsurvey}
\bysame, \emph{Heegaard diagrams and {F}loer homology}, International
  {C}ongress of {M}athematicians. {V}ol. {II}, Eur. Math. Soc., Z\"urich, 2006,
  pp.~1083--1099.

\bibitem[Pet09]{Petkova}
Ina Petkova, \emph{Cables of thin knots and bordered {H}eegaard {F}loer
  homology}, preprint (2009), \ arXiv:0911.2679v1.

\bibitem[Ras03]{R}
Jacob Rasmussen, \emph{Floer homology and knot complements}, Ph.D. thesis,
  Harvard University, 2003.

\end{thebibliography}

\end{document}